\documentclass{amsart}
\usepackage{fullpage,amsfonts,amsmath,amscd,amssymb}
\input{xy}
\xyoption{all}
\xyoption{curve}

\newtheorem{theorem}{Theorem}[section]
\newtheorem{lemma}[theorem]{Lemma}
\newtheorem{prop}[theorem]{Proposition}
\newtheorem*{con}{Conjecture}
\newtheorem{cor}[theorem]{Corollary}
\newtheorem*{hyp}{Hypothesis}

\DeclareMathOperator{\Res}{res}
\DeclareMathOperator{\Cor}{cor}
\newcommand{\Mod}{{\mathrm M}{\mathrm o}{\mathrm d}}
\newcommand{\End}{{\mathrm E}{\mathrm n}{\mathrm d}}
\newcommand{\Fun}{{\mathrm F}{\mathrm u}{\mathrm n}}
\newcommand{\Hom}{{\mathrm H}{\mathrm o}{\mathrm m}}
\newcommand{\PGL}{{\mathrm P}{\mathrm G}{\mathrm L}}
\newcommand{\GL}{{\mathrm G}{\mathrm L}}
\newcommand{\Irr}{{\mathrm I}{\mathrm r}{\mathrm r}}
\newcommand{\Coh}{{{\mathrm C}{\mathrm o}{\mathrm h}}}
\newcommand{\CG}{\Coh_G}
\newcommand{\CGc}{\Coh_{G^\chi}}
\newcommand{\Rep}{{\mathrm R}{\mathrm e}{\mathrm p}}

\newcommand{\sS}{\mathcal{S}}
\newcommand{\sV}{\mathcal{V}}
\newcommand{\sB}{\mathcal{B}}
\newcommand{\sD}{\mathcal{D}}
\newcommand{\sC}{\mathcal{C}}
\newcommand{\sG}{\mathcal{G}}
\newcommand{\sM}{\mathcal{M}}
\newcommand{\sN}{\mathcal{N}}

\newcommand{\mM}{\mathfrak{M}}
\newcommand{\mg}{\mathfrak{g}}
\newcommand{\bB}{{\mathbb B}}
\newcommand{\bC}{{\mathbb C}}
\newcommand{\bK}{{\mathbb K}}
\newcommand{\bF}{{\mathbb F}}
\newcommand{\bM}{{\mathbb M}}

\newcommand{\bZ}{{\mathbb Z}}

\newcommand{\bBM}{{\mathbb B}^\mu}
\newcommand{\bP}{{\mathbb B}^\Phi}

\newcommand{\fg}{{\mathfrak{g}}}

\newenvironment{mylist}{\begin{list}{}{
\setlength{\itemsep}{0mm}
\setlength{\parskip}{0mm}
\setlength{\topsep}{1mm}
\setlength{\parsep}{0mm}
\setlength{\itemsep}{0mm}
\setlength{\labelwidth}{6mm}
\setlength{\labelsep}{3mm}
\setlength{\itemindent}{0mm}
\setlength{\leftmargin}{9mm}
\setlength{\listparindent}{6mm}
}}{\end{list}}

\begin{document}



\title{Generalised Burnside Rings, $G$-categories and Module Categories}
\author{Paul Gunnells}
\address{Department of Mathematics and Statistics, University of Massachusetts, Amherst, MA 01003, USA}
\email{gunnells@math.umass.edu}
\author{Andrew Rose}
\address{Mathematics Department, University of Warwick, Coventry, CV4 7AL, UK}
\email{Andrew1145@gmail.com}
\author{Dmitriy Rumynin}
\address{Department of Mathematics, University of Warwick, Coventry, CV4 7AL, UK}
\email{D.Rumynin@warwick.ac.uk}
\thanks{The second author was supported by the EPSRC. 
The third author would like to thank
Isaac Newton Institute of Mathematical Sciences where this research has been started
and Max Planck Institute in Bonn where it has been completed.}

\date{December 1, 2011}

\begin{abstract}
This note describes an application of the theory of generalised Burnside rings to algebraic representation theory.  
Tables of marks are given explicitly for the groups $S_4$ and $S_5$ 
which are of particular interest in the context of reductive algebraic groups. 
As an application, the base sets for the nilpotent element  $F_4 (a_3)$ are computed.
\end{abstract}

\maketitle

Our aim is to combine two modern lines of enquiry. 
The first line is generalised Burnside rings which were recently introduced by 
Hartmann and Yal\c{c}in \cite{HartmannYalcin07}.
The second line is the study of tensor categories attached to cells in affine Weyl groups
by
Bezrukavnikov, Finkelberg and Ostrik \cite{Bezruk01,BezrukFO06}. We show how one can use generalised Burnside rings to carry through explicit 
calculations with module categories.

The note is organised as follows. 
In section 1 we introduce generalised Burnside rings. 
Our generalised Burnside ring is slightly more general 
than the one of Hartmann and Yal\c{c}in.
We define it for a general functor rather than the cohomology functor.
For our applications, the most crucial functor 
is the Schur multiplier $\mu (G)$, so 
we describe the table of marks for the Schur multiplier for the symmetric groups
$S_4$ and $S_5$. 
In section 2 we 
discuss the connection between $\mu$-decorated sets and $G$-algebras.
In section 3 we 
discuss the connection between $\mu$-decorated sets and groupoids.
In section 4 we study module categories in the spirit of Bezrukavnikov and Ostrik \cite{Bezruk01}.
In section 5 
we investigate base sets of Kazhdan-Lusztig cells \cite{Lus}.
We use a computer calculation with Kazhdan-Lusztig polynomials and a pen-and-paper calculation
in the Burnside ring of $S_4$ to determine the base set of the largest finite double cell
in the affine Weyl group of the type $F_4$.  
In the final section 6 we explain 
an application to representation theory
of the reduced enveloping algebra $U_\chi (\fg)$ where $\fg$ is of the type $F_4$
and $\chi$ is of the type $F_4(a_3)$.

The authors would like to thank  
M. Belolipetsky, M. Finkelberg, 
S. Goodwin, J. Humphreys, G. Rohrle
and W. Soergel
for stimulating discussions.
The second author was supported by the EPSRC. 
The third author would like to thank
Isaac Newton Institute of Mathematical Sciences where this research has been started
and Max Planck Institute in Bonn where it has been completed.

\section{Generalised Burnside ring}
\label{s:intro}

Let $G$ be a finite group, $\sS (G)$ its category of subgroups.
Objects of $\sS (G)$ are subgroups of $G$. 
The morphisms $\sS(A,B)$ are conjugations $\gamma_x: A \rightarrow B$, 
$\gamma_x (a) = xax^{-1}$, $x\in G$ whenever $xAx^{-1}\subseteq B$, restricted to $A$.
Thus, $\gamma_x$ and $\gamma_y$ define the same morphism in $\sS(A,B)$
whenever $y^{-1}x$ is in the centraliser of $A$.
The composition of morphisms is the composition of homomorphisms.

{\em A generalised Burnside ring} $\bP_R (G)$ 
depends on  a contravariant functor $\Phi$  from $\sS(G)$ to the category of semigroups
and  a commutative ring of coefficients $R$.
As an $R$-module it is generated by disjoint union of all $\Phi (A)$, $A\in\sS(G)$. 
We write $\langle a,A\rangle $ for an element of the semigroup $a\in\Phi (A)$.
The $R$-module generators satisfy the relations
$$
\langle a,A\rangle  = \langle \Phi(\gamma_g)(a), g^{-1}Ag \rangle 
$$
for all $g\in G$, $A\in \sS(G)$, $a\in \Phi (A)$. 
Notice that $\langle a,A\rangle  + \langle b, A\rangle  \neq \langle ab,A\rangle $ in general (we think of semigroups as multiplicative semigroups). 
The multiplication is $R$-bilinear, defined on the $R$-module generators by the formula
$$
\langle a,A\rangle  \cdot \langle b, B\rangle  = \sum_{AxB\in A\backslash G/B} 
\langle 
\Phi(\gamma_1: A\cap xBx^{-1} \rightarrow A)(a) 
\Phi(\gamma_{x^{-1}} : A\cap xBx^{-1} \rightarrow B)(b)
,A\cap xBx^{-1}\rangle .
$$
\begin{lemma}
Defined as above, $\bP_R (G)$ is an associative $R$-algebra. 
If $\Phi$ is a functor to monoids then $\bP_R (G)$ is unitary.
\end{lemma}
\begin{proof}
A sleek way to prove this is to interpret $\bP_R(G)$ 
as a Grothendieck group of $\Phi$-decorated $G$-sets. 
A $\Phi$-decorated $G$-set is a finite set $X$ 
with a $G$-action and {\em a frill} 
$\pi_x \in \Phi (G_x)$ attached to each point $x\in X$. 
Here $G_x$ is
the stabiliser of $x$ in $G$.
The frills $\pi_x$ must be equivariant in a sense that 
$\pi_{gx} = \Phi (\gamma_g) (\pi_{x})$.

The element  $\langle a,A\rangle $ represents a homogeneous set $G/A$ with frills $\pi_{gA} = \Phi (\gamma_g) (a)$.
The addition corresponds to disjoint union $[X]+[Y]=[X\coprod Y]$ and 
the multiplication corresponds to the direct product  
$[X]\cdot[Y]=[X\times Y]$,
where the frills multiplied in the corresponding semigroup (note that $G_{(x,y)}= G_x\cap G_y$):
$$
\pi_{(x,y)} = 
\Phi(\gamma_1 : G_{(x,y)} \rightarrow G_x)(\pi_x) 
\Phi(\gamma_1 : G_{(x,y)} \rightarrow G_y)(\pi_y). 
$$ 

If $\Phi$ is a functor to monoids, then $\langle 1,G\rangle$ 
is the identity of  $\bP_R(G)$ as can be easily verified.
\end{proof}

The subgroup category $\sS (G)$ is an example of {\em a fusion system}.
Burnside rings of fusion systems were constructed by Diaz and Libman \cite{DiLi}.
Generalised Burnside rings can be extended to fusion systems as well.
An interested reader is invited to follow this lead, especially if the reader
can think of useful applications.

The notion of {\em a mark homomorphism}
can be extended to generalised Burnside rings (cf. \cite[\S 6]{HartmannYalcin07}). 
Let $S$ be an associative $R$-algebra, 
$\alpha : \Phi (A) \rightarrow S^\times$ 
a semigroup homomorphism for some $A\in\sS(G)$. The corresponding mark is an $R$-linear map
$f_A^\alpha : \bP_R(G) \rightarrow S$ given by the formula
\begin{eqnarray}\label{eqn:hom}
f_A^\alpha (\langle b, B \rangle)&=&\frac{1}{|B|}\sum_{g\in X}
\alpha (
\Phi(\gamma_{g} : A \rightarrow B)(b) 
)
\end{eqnarray}
where $X=\{g\in G\mid gAg^{-1}\subseteq B \}$. 

\begin{lemma}\label{lem:f_i}
The mark $f_A^\alpha$ is an $R$-algebra homomorphism. It is unitary if $\Phi$ is a functor to monoids and $\alpha$ is unitary.
\end{lemma}
\begin{proof} 
Let us reinterpret the mark using $\Phi$-decorated sets.
The condition $gAg^{-1}\subseteq B$ means that $Ag^{-1}B=g^{-1}B$,
i.e., $A$ lies in the stabiliser of $g^{-1}B$. The frill of $X$ with $[X]=<b,B>$ at $g^{-1}B$ 
is $\Phi (\gamma_g)(b)$. Thus, on the level of decorated sets,
\begin{eqnarray}\label{eqn:hom1}
f_A^\alpha ( [X,\pi_x ]) &=&
\sum_{x\in X^A}
\alpha (
\Phi(\gamma_{1} : A \rightarrow G_x)(\pi_x) 
)
\end{eqnarray}
and, consequently,
$$
f_A^\alpha
( [(X,\pi_x)\times (Y,\psi_y) ])
= 
\sum_{(x,y)\in (X\times Y)^A} 
\alpha \Big(
\Phi(\gamma_1: A \rightarrow G_x)(\pi_x)
\Phi(\gamma_1: A \rightarrow G_y)(\psi_y)
\Big)
=
$$
$$
\sum_{x\in X^A} \sum_{y\in Y^A}
\Big( \alpha (
\Phi(\gamma_1: A \rightarrow G_x)(\pi_x))
\Big)
\Big( 
\alpha (\Phi(\gamma_1: A \rightarrow G_y)(\psi_y))
\Big)
=
f_A^\alpha
(X,\pi_x)
f_A^\alpha
(Y,\psi_y)
$$

In the unitary case, the identity of $\bP_R(G)$ is $\langle 1,G\rangle$ and
$
f_A^\alpha (\langle 1, G \rangle) =
\alpha (\Phi(\gamma_1)(1)) = \alpha (1_{\Phi (A)}) = 1_S
.$
\end{proof}

Note that if $\Phi(A)$ is a finite abelian group there is an isomorphism between the group of linear characters of $\Phi(A)$ and the group $\Phi(A)$.
If all $\Phi (A)$ are finite abelian groups then
the number of distinct marks 
is equal to the rank of $\bP_R (G)$ over $R$. Let us formulate this as a corollary.

\begin{cor}
Suppose all $\Phi (A)$ are finite abelian groups, $N$ the least common multiple of all the orders of elements in all $\Phi (A)$.
If $R$ is a field containing primitive $N$-th root of unity then the mark homomorphisms define an isomorphism
$\bP_R (G) \rightarrow \oplus R$.
\end{cor}

Before formulating the next property, let us introduce the notion of {\em the dual set}.
Let $Y$ be a $\Phi$-decorated set such that each frill $\pi_m \in \Phi(G_m)$ is invertible.
The dual set $Y^\vee$ 
has the same underlying $G$-set $Y$ but the frills are inverted:
each $\pi_m \in \Phi (G_m)$ is replaced with $\pi_m^{-1}$. 

\begin{lemma}\label{com_inv}
If $\Phi (A)$ is abelian for each $A\leq G$ then
$\bP_R (G)$ is a commutative ring.
If $\Phi (A)$ is a group for each $A\leq G$ then
$\bP_R (G)$ is a ring with involution.
\end{lemma}
\begin{proof}
The involution is defined by $[Y]^\vee := [Y^\vee]$.
Now both statements follow from the definition of  $\bP_R (G)$. 
\end{proof}

If $R=\bZ$, we write $\bP (G)$ for $\bP_R (G)$.
Several functors $\Phi$ are interesting for applications.
First of all, the trivial functor $\Phi (H) = \{ 1\}$ gives the classical Burnside ring $\bB (G)$,
the Grothendieck ring of finite $G$-sets.
Another interesting functor is  $\Phi (H) = \Rep_+ (H)$, the effective part of the representation ring of $H$ over $\bZ$.
It has two different semigroup structures, corresponding to tensor products or direct sums of representations.
The corresponding Burnside ring $\bP (G)$ is
the Grothendieck ring of pairs $(X,V)$, a finite $G$-set and a $G$-equivariant vector bundle on it.
Another interesting functor is  the effective part of Burnside ring itself $\Phi (H) = \bB_+ (H)$.
Again it has two different semigroup structures, corresponding to products or unions. 
The corresponding Burnside ring $\bP (G)$ is
the Grothendieck ring of fibred $G$-sets $Y\rightarrow X$, i.e. surjective maps of $G$-sets, where one considers $Y$
as an equivariant fibration over $X$.
Hartmann and Yal\c{c}in have studied
$\Phi (H) = H^\ast (H , M)$ and $\Phi (H) = H^n (H , M)$, where $M$ is a $G$-module \cite{HartmannYalcin07}.
They have called the corresponding $\bP (G)$ a cohomological Burnside ring.

The second cohomological Burnside ring is of particular interest to us.
It will be studied for the rest of the paper.
Namely, if $\bK$ is a field, we need the functor  
$\mu_\bK (H) = H^2 (H ,\bK^\times)$
where $H$ acts trivially on the multiplicative group $\bK^\times$ of the field.
As soon as  $\bK^\times$ has enough torsion, say $\bK$ admits a $|G|$-th primitive root of $1$ (for instance, if $\bK$ is algebraically closed
of characteristic $p$ not dividing $|G|$),
$\mu_\bK (H)$ is the Schur multiplier of $H$
\cite{Karpilovsky87}.
In particular, it is independent of $\bK$ and will be denoted just $\mu (H)$ with the corresponding Burnside ring $\bBM (G)$.

We present the tables of marks for $\bB^\mu (G)$ of the symmetric groups $S_4$ and $S_5$ in Tables~\ref{tab:S4} and~\ref{tab:S5}. 
We use the notation 
$\langle K \rangle = \langle 1 , K\rangle$,
$\langle K^\prime \rangle = \langle x , K\rangle$
where $x$ is a generator of $C_2$, the only possible nontrivial $\mu (H)$, 
and
$f^\prime_H$ is a mark with nontrivial character of $C_2$.
In these tables $D_8$ and $D_{10}$ denote the standard dihedral group of orders 8 and 10, 
$K_1=\langle(12),(34)\rangle$ and $K_2=\langle(12)(34),(13)(24)\rangle$
denote nonconjugate Klein four groups,
$C_n$ denotes a cyclic subgroup of order $n$ generated by a single cycle.
Notation $H_n$ is reserved for various non-standard subgroups of order $n$: 
$H_2$ is generated by $(1,2)(3,4)$,
$H_{20}$ is the normaliser of $C_5$ in $S_5$,
$H_{6}:= \langle(123),(12)(45)\rangle$ is a nonstandard $S_3$.   
The columns of the tables correspond to values of the marks $f_H$ or $f^\prime_H$ 
ordered as for the rows. 
Appended to the tables are the values of the equivariant Euler characteristic
${\mathfrak M}:\bBM(G)\rightarrow \mathbb{Z}$. 
It will be defined 
in Section~\ref{sec_mod_cat}.
Notice that $\bB^{\mu_\bK}$ over any field $\bK$ 
of characteristic not 2 will have the same table of marks.

The tables were computed by lifting 
data from the ordinary table of marks and the following lemma.
\begin{lemma}
Let $H\leq K\leq G$, $\bK$ a field of characteristic not 2.  Suppose that $|\mu(K)|=|\mu(H)|= 2$ and 
2 does not divide the index $|K:H|$. 
Then $f^\prime_H(\langle K^\prime \rangle)=-f_H (\langle K \rangle)$.
\end{lemma}
\begin{proof}
Let $\tau\in \mu(K)$ and $\nu\in \mu(H)$ be non-trivial cocycles. 
The corestriction map on cocycles satisfies $\Res_{K,H}(\Cor_{H,K}(\nu))=|K:H|\nu$~\cite[Ch.~1]{Karpilovsky87}.  
Thus, $\Res_{K,H}(\Cor_{H,K}(\nu))=\nu$. 
Therefore $\nu$ corestricts to the nontrivial cocycle $\tau$ and $\Res_{K,H}(\tau)\neq 1$. 
The lemma now follows from Equation~(\ref{eqn:hom}).
\end{proof}

\section{$G$-algebras and $\mu_\bK$-decorated sets}

A $G$-algebra is an associative algebra $A$ with a (left) action of $G$. 
As a default option, an action is always a left action. However, right actions
often appear naturally. For instance, 
the group $G$ acts (on the right) on the abelian category $A-\Mod$
of left $A$-modules.

We say that $G$ has a right action on a category $\sC$ if for every 
$g\in
G$, we have an autoequivalence, 
$[g]:\sC \rightarrow \sC$, together with
natural isomorphisms $\gamma_{g,h}:[g]\circ [h] \rightarrow [hg]$ such that $[1]$ is the identity functor.
In this case, we call $\sC$ a $G$-category.

Sometimes in the
literature such actions are called ``weak'' as opposed to
``strong'' actions, which satisfy 
commutativity of the associativity constraint diagrams
$$
\begin{CD} 
{[f]\circ[g]\circ [h]} @> {\gamma_{f,g}} >> {[gf]\circ [h]} \\ 
@V{\gamma_{g,h}} VV @VV {\gamma_{gf,h}} V \\ 
{[f]\circ[hg]} @> {\gamma_{f,hg}} >> {[hgf]} \\ 
\end{CD} 
$$
for all $f,g,h\in G$.
Here we are
not interested in associativity constraints.
%
%

Let us describe  $[g]$ and $\gamma_{g,h}$ for $\sC=A-\Mod$, in detail. On objects, $M^{[g]}=M$ with the new action
of $A$ given by $a\cdot^{[g]} m = g(a)m$. On morphisms, $f^{[g]}=f$.
Finally, for each object $M$, $\gamma_{g,h}(M): (M^{[h]})^{[g]} \rightarrow M^{[hg]}$
is the identity map. Notice that
$a\cdot^{[h][g]} m = g(a)\cdot^{[h]} m = h(g(a))m=a\cdot^{[hg]} m$.
Notice further that this action is strong.

Going back to a general $G$-category, we say that an object $X$ is {\em equivariant} if
all its twists $X^{[g]}$ are isomorphic to $X$ and there exists a system of isomorphisms
$\alpha_g : X\rightarrow X^{[g]}$ such that  the diagrams
$$
\begin{CD} 
{X} @> {\alpha_{h}} >> {X^{[h]}} \\ 
@V{\alpha_{gh}} VV @VV {(\alpha_{g})^{[h]}} V \\ 
{X^{[gh]}} @< {\gamma_{g,h}(X)} << {X^{[g][h]}} \\ 
\end{CD} 
$$
are commutative for all $g,h\in G$. This notion allows us to characterise $A\ast G$-modules among
$A$-modules where
$A\ast G$ is the skew group algebra, i.e. a free left $A$-module with a basis $G$ and a multiplication coming from those of $A$ and $G$ with 
an additional rule 
$ga= g(a) g$ for all  $a\in A$, $g\in G$.
\begin{lemma}
\label{G-eq}
An $A$-module $M$ is an equivariant object of $A-\Mod$ 
if and only if 
it admits a structure of an $A\ast G$-module.
\end{lemma}
\begin{proof}
The connection between the equivariant structure and the action of $G$ is given by 
$\alpha_g (m) = g\cdot m$.
One can verify that the two sets of axioms are equivalent.
\end{proof}

A functor 
$\Phi : \sC \rightarrow \sD$ between $G$-categories is a $G$-functor
if it is equipped with a system of natural isomorphisms 
$$
\beta_g : \Phi \circ [g]_\sC\rightarrow [g]_\sD \circ \Phi \, , \ \  
g\in G
$$ such that the square
$$
\begin{CD} 
{\Phi(X^{[g]})} @> {\beta_{g}(X)} >> {\Phi(X)^{[g]}} \\ 
@V{\Phi(t^{[g]})} VV @VV {\Phi(t)^{[g]}} V \\ 
{\Phi(Y^{[g]})} @> {\beta_{g}(Y)} >> {\Phi(Y)^{[g]}} \\ 
\end{CD} 
$$
is commutative for all $t\in\sC(X,Y)$, $g\in G$ and the pentagon

\begin{picture}(300,120)(0,0)
\put(0,80){$\Phi(X^{[gh]})$}
\put(110,80){$\Phi(X^{[g][h]})$}
\put(230,80){$\Phi(X^{[g]})^{[h]}$}
\put(100,83){\vector(-1,0){60}}
\put(160,83){\vector(1,0){60}}
\put(50,87){{\small $\Phi(\gamma_{g,h}(X))$}}
\put(167,87){{\small $\Phi(\beta_{h}(X^{[g]}))$}}
\put(0,20){$\Phi(X)^{[gh]}$}
\put(235,20){$\Phi(X)^{[g][h]}$}
\put(15,70){\vector(0,-1){40}}
\put(245,70){\vector(0,-1){40}}
\put(20,50){{\small $\beta_{gh}(X)$}}
\put(203,50){{\small $\beta_{g}(X)^{[h]}$}}
\put(225,23){\vector(-1,0){178}}
\put(100,30){{\small $\gamma_{g,h}(\Phi (X))$}}
\end{picture}
\newline
is commutative for all objects $X\in\sC$ and $g,h\in G$.
A $G$-equivalence is a $G$-functor which is an equivalence.

\begin{lemma}
\label{G-eq_func}
Let  $\Phi : \sC \rightarrow \sD$ be a $G$-equivalence between $G$-categories.
If $X$ is a $G$-equivariant object in $\sC$ then $\Phi (X)$ is a $G$-equivariant object in $\sD$.
\end{lemma}
\begin{proof}
Let $X=(X,\alpha_g)$ be an equivariant object.
The equivariant structure on $\Phi(X)$ is given by 
the compositions
$\beta_g(X) \circ \Phi (\alpha_g) : \Phi (X) \rightarrow \Phi (X^{[g]}) \rightarrow \Phi (X)^{[g]}$.
To verify the axiom we analyse the following diagram.

\begin{picture}(300,170)(0,0)
\put(0,140){$\Phi(X)$}
\put(110,140){$\Phi(X^{[h]})$}
\put(230,140){$\Phi(X)^{[h]}$}
\put(40,143){\vector(1,0){60}}
\put(160,143){\vector(1,0){60}}
\put(57,147){{\small $\Phi(\alpha_{h})$}}
\put(170,147){{\small $\beta_{h}(X)$}}
\put(15,130){\vector(0,-1){40}}
\put(125,130){\vector(0,-1){40}}
\put(245,130){\vector(0,-1){40}}
\put(20,110){{\small $\Phi(\alpha_{gh})$}}
\put(130,110){{\small $\Phi(\alpha_g^{[h]})$}}
\put(200,110){{\small $\Phi(\alpha_g)^{[h]}$}}
\put(0,80){$\Phi(X^{[gh]})$}
\put(110,80){$\Phi(X^{[g][h]})$}
\put(230,80){$\Phi(X^{[g]})^{[h]}$}
\put(100,83){\vector(-1,0){60}}
\put(160,83){\vector(1,0){60}}
\put(50,87){{\small $\Phi(\gamma_{g,h}(X))$}}
\put(167,87){{\small $\beta_{h}(X^{[g]})$}}
\put(0,20){$\Phi(X)^{[gh]}$}
\put(235,20){$\Phi(X)^{[g][h]}$}
\put(15,70){\vector(0,-1){40}}
\put(245,70){\vector(0,-1){40}}
\put(20,50){{\small $\beta_{gh}(X)$}}
\put(203,50){{\small $\beta_{g}(X)^{[h]}$}}
\put(225,23){\vector(-1,0){178}}
\put(100,30){{\small $\gamma_{g,h}(\Phi (X))$}}
\end{picture}

The top left square is commutative because $X$ is equivariant. The top right square and the bottom pentagon are commutative
because $\Phi$ is a $G$-functor.
Thus, the whole diagram is commutative for all $g,h\in G$. It remains to notice that the outer edges of the diagram read off
the equivariance condition for $\Phi (X)$.
\end{proof}

We say that two $G$-algebras $A$ and $B$ are {\em $G$-Morita equivalent} 
if 
there exists a $G$-equivalence
$\Phi : A-\Mod \rightarrow B-\Mod$. 
We say that a Morita context $(A,B, \,_AM_B, \,_BN_A, \phi, \psi)$ is {\em nondegenerate}
if $\phi$ and $\psi$ are isomorphisms.
We say it
is {\em $G$-equivariant}
if  
\begin{mylist}
\item[(1)] both $M$ and $N$ are $G$-modules,
\item[(2)] $g\cdot (amb) = ( g\cdot a)(g \cdot m)(g\cdot b)$
for all $a\in A$, $b\in B$, $g\in G$, $m\in M$,
\item[(3)] 
$g\cdot (bna) = ( g\cdot b)(g \cdot n)(g\cdot a)$ for all $a\in A$, $b\in B$, $g\in G$, $n\in N$,
\item[(4)] the bimodule maps
$\phi : M\otimes_BN \rightarrow A$ and $\psi : N\otimes_AM \rightarrow B$ 
are homomorphisms of $G$-modules.
\end{mylist}
The following theorem characterises $G$-Morita equivalences within the context of Morita theory.

\begin{theorem}
\label{G_morita}
Associative $G$-algebras
$A$ and $B$ are $G$-Morita equivalent if and only if there exists a nondegenerate $G$-equivariant 
Morita context $(A,B, \,_AM_B, \,_BN_A, \phi, \psi)$.
\end{theorem}
\begin{proof}
A nondegenerate $G$-equivariant context gives a $G$-equivalence $\Phi : A-\Mod \rightarrow B-\Mod$
by
$\Phi (P) = N\otimes_AP$ with an inverse equivalence $T \mapsto M\otimes_B T$.
The equivariant structure on $\Phi$ is given by 
$$
N\otimes_AP^{[g]}
\rightarrow
(N\otimes_AP)^{[g]}\, , \ \ \ 
n\otimes p \mapsto
g\cdot n \otimes p
.
$$
Commutativity of the squares and the pentagons is obvious.

In the opposite direction, let $\Phi: A-\Mod\rightarrow B-\Mod$ be a $G$-equivalence 
and
$\Psi: B-\Mod\rightarrow A-\Mod$
its inverse $G$-equivalence.
Out of this one derives a standard nondegenerate Morita context:
$N=\Phi (A)$, 
$M=\Psi (B)$.
As $A$ and $B$ are progenerators, the functor $\Phi$ is naturally isomorphic to $N\otimes_A$
and $\Psi$ is naturally isomorphic to $M\otimes_B$.
The isomorphisms
$\phi : M\otimes_BN \cong \Psi (\Phi (A)) \rightarrow A$ and 
$\psi : N\otimes_AM \cong \Phi (\Psi (B)) \rightarrow B$ 
come from the natural isomorphisms.

It remains to see through the action of $G$.
The object $N=\Phi (A)$ is $G$-equivariant 
by Lemma~\ref{G-eq_func}, i.e., it is naturally a $B*G$-module by
Lemma~\ref{G-eq}.
Thus, $g\cdot (bn) = ( g\cdot b)(g \cdot n)$
for all $b\in B$, $g\in G$, $n\in N$.
Since $\Phi$ is an equivalence of categories, $\End_B(N)\stackrel{\Phi}{\cong} \End_A(A)\cong A$,
and $N$ is a $B$-$A$-bimodule. Finally, the property $g\cdot (na) = ( g\cdot n)(g \cdot a)$
for all $a\in A$, $g\in G$, $n\in N$ follows from the same property for $A$. 
To prove this, observe that  if $R_a$ is a right multiplication by $a$ then the property for $A$ manifests
in the diagram
$$
\begin{CD} 
{A} @> {\alpha_{g}} >> {A} \\ 
@V{R_a} VV @VV {R_{g(a)}} V \\ 
{A} @> {\alpha_{g}} >> {A} \\ 
\end{CD} 
$$
being commutative (N.B., $A^{[g]}=A$). Applying $\Phi$ gives commutativity of the left square in the diagram
$$
\begin{CD} 
{N} @> {\Phi(\alpha_{a})} >> {N} @> {\beta_{g}(N)} >> {N} \\ 
@V{R_a} VV @V{R_{g(a)}} VV @VV {R_{g(a)}} V \\ 
{N} @> {\Phi(\alpha_{g(a)})} >> {N} @> {\beta_{g(a)}(N)} >> {N} \\ 
\end{CD} 
$$
(N.B., $\Phi (R_a)=R_a=R_a^{[g]}$). The right square is commutative by the definition of a $G$-functor.
Thus, the whole diagram is commutative that manifests in $g\cdot (na) = ( g\cdot n)(g \cdot a)$
for all $a\in A$, $g\in G$, $n\in N$.

Similarly, $M=\Psi (B)$ is an $A$-$B$-module with a compatible action of $G$.
The bimodule isomorphisms
$\phi : M\otimes_BN \rightarrow A$ 
and $\psi : N\otimes_AM \rightarrow B$  
come from the isomorphisms $\Psi (\Phi (A))\cong A$ and $\Phi (\Psi (B))\cong B$.
The latter are isomorphisms of $G$-modules. Hence, so are $\phi$ and $\psi$.
\end{proof}


Every $G$-algebra $A$ over $\bK$ 
admits a canonical $\mu_\bK$-decorated set $\Irr (A)$ of isomorphism classes
of absolutely simple $A$-modules. 
Recall that a simple $A$-module $M$ is absolutely simple if $\End_A(M)=\bK$.
The (left) action of $G$ on $\Irr (A)$ comes from the (right) action 
on the category $A-\Mod$:
$g\cdot [M] = [M^{[g^{-1}]}]$.

Let us observe the cocycle.
Let $G_M$ be the stabiliser of $[M]\in\Irr (A)$, $a\in \Hom_\bK (A\otimes M, M)$
the $A$-action on $M$.
Since $G_M$ does not change the isomorphism class of the module, $G_Ma \subseteq \GL (M)a$.
The stabiliser of $a$ in $\GL(M)$ is the group of module automorphisms of $M$, 
which is $\bK^\times$ since $M$ is absolutely irreducible.
Hence, $X \mapsto X\cdot a$ is a bijection from the group $\PGL(M)$ to the orbit $\GL(M)a$.
Thus, $g\mapsto g^{-1}\cdot a$ defines a natural function
$\phi_M: G_M \rightarrow \PGL (M)$. 
This function is a group homomorphism
because the actions of $G_M$ and $\PGL(M)$ commute.
Indeed, the action of $G_M$ factors through
$\GL (A)$, while $\GL (A)$ and $\PGL (M)$ act on the different tensor components of $\Hom_\bK (A\otimes M, M)$. Hence,
$$
\phi (gh) \cdot a = (gh)^{-1} \cdot a = h^{-1} \cdot (g^{-1}\cdot a) = h^{-1} \cdot (\phi (g) \cdot a) = \phi (g) \cdot ( h^{-1} \cdot a) = \phi (g) \phi(h) \cdot a.
$$
The obstruction to lifting of $\phi_M$ to a homomorphism $G_M \rightarrow \GL (M)$
is a cocycle $\theta_M \in Z^2(G_M, \bK^\times)$, well defined up to a coboundary.
Thus, the frill $\pi_M := [\theta_M] \in \mu_\bK (G_M)$ and
$\Irr (A)$ is a $\mu_K$-decorated $G$-set, albeit it does not have to be finite for an arbitrary $A$.

\begin{theorem}
\label{morita_decor}
The function $\Upsilon ([A]) = [\Irr (A)]$ is a bijection from 
the set of $G$-Morita equivalence classes of semisimple split $G$-algebras to
the set of isomorphism classes of finite $\mu_\bK$-decorated $G$-sets.
Moreover, using the multiplication in $\bBM (G)$,
$$
\Upsilon ([A\otimes B]) = \Upsilon ([A])\Upsilon([B]),
\ 
\Upsilon ([A\oplus B]) = \Upsilon ([A])+ \Upsilon([B])
\ \mbox{ and } \
\Upsilon ([A^{op}]) = 
\Upsilon ([A])^\vee .
$$ 
for all semisimple split $G$-algebras $A$ and $B$.
\end{theorem}
\begin{proof}
To prove bijectivity 
we describe the inverse function $\Upsilon^{-1}$.
Let $X$ be a finite $\mu_\bK$-decorated $G$-set,
$X_0\subseteq X$ a set of representatives of $G$-orbits.
For each point $m\in X_0$ 
let us choose an irreducible projective representation $V_m$ of $G_m$
that affords the frill $\pi_m$.
Let $T_m$ be the right transversal of $G_m$ in $G$.
Now, for each $x\in X$ there exist unique $m\in X_0$, $g\in T_m$ such that
$x=g\cdot m$.
We define a projective representation $V_x$ of $G_x$ by
$$
V_x = V_m, \ \ h\cdot v := (g^{-1}hg)\cdot v, \ \ \forall h\in G_x, \ v\in V_x=V_m.
$$
The collection $\sV=(V_x,x\in X)$ of vector spaces is a $G$-equivariant vector bundle on $X$ \cite{Bezruk01}.
In plain terms, it means that there are linear maps 
$
\Theta_x (g): V_x \rightarrow V_{g\cdot x} 
$
for all $g\in G$, $x\in X$ such that
$\Theta_{gx} (h)\Theta_x (g)=\Theta_x (hg)$
and
$\Theta_x (1)=I_{V_x}$.
To see them, observe a bijection between $\sV$ and the fibre product
$$
\coprod_{m\in X_0} G\times_{G_m} V_m :=
\coprod_{m\in X_0} G\times V_m \Big/_\sim 
\stackrel{\cong}{\longrightarrow} \sV
$$
where $(g,v)\sim (g^\prime, v^\prime)$ if and only if they are in the same $G\times V_m$ and
there exists
$h\in G_m$ such that $g^\prime = gh$, $v^\prime = h^{-1}v$.
Now $\Theta_x (g) ([h,v]) = [gh,v]$.
Using this, we can construct a semisimple split $G$-algebra 
$$
A:= \oplus_{x\in X} \End_\bK (M_x) 
, \ \ g\cdot (\alpha_x) = ( \Theta_x(g) \alpha_x \Theta_{gx}(g^{-1}))
$$
with $\Irr (A)$ isomorphic to $X$ as $\mu_\bK$-decorated $G$-sets.
Notice that the different choice of $X_0$ or one of $T_m$ will lead to an isomorphic algebra,
while a different choice of one of $V_m$ will lead to a $G$-Morita equivalent algebra. Thus $\Upsilon$ is a bijection.

The first two properties of $\Upsilon$ are immediate.
The last property follows from the fact that the simple $A^{op}$-modules
are the dual spaces $M^*$ of simple $A$-modules $M$. The cocycle of $G_M$-action
on $M^\ast$ is $\pi_M^{-1}$.
\end{proof}

Theorem~\ref{morita_decor}
gives a new presentation of the Burnside ring $\bBM_R(G)$. As a left $R$-module it is generated
by $G$-Morita equivalence classes of semisimple split $G$-algebras subject to relations
$$
[A\oplus B] = [A]  +   [B]
$$
while the multiplication is given by the rule
$$
[A]\cdot [B] = [A\otimes B] .
$$

We finish this section outlining the role of generalised Burnside rings in number theory.
A similar construction for  the usual Burnside rings 
have  recently been used by Dockchitsers to prove a partial case of the parity conjectures \cite{DoDo}.
%

Let $\bF \leq \bK$ be a $G$-Galois extension of algebraic number fields. 
Let us consider 
a central simple $n^2$-dimensional algebra $S$ over $\bK^H$, split over $\bK$,
where $H$ is a subgroup of $G$.
The algebra $S$
is uniquely determined up to an isomorphism 
by its system of factors 
$\alpha_S \in H^1 (H, \PGL_n (\bK))$.
The long exact sequence in nonabelian cohomology gives 
an embedding $H^1 (H, {\PGL_n} (\bK)) \hookrightarrow H^2 (H, \bK^\times)$. 
Thus, we can think that  $\alpha_S \in H^2 (H, \bK^\times)$.
Then nonisomorphic algebras $S$ can have the same $\alpha_S$. 
By Artin-Wedderburn's theorem, $S\cong M_k (D_S)$ where $D_s$ is a simple central division algebra.
Then $\alpha_S = \alpha_T$ if and only if $D_S \cong D_T$. 


Now
we can interpret  $\langle a,H\rangle \in \bBM (G)$ 
as a Morita equivalence class $[A]$ of a simple $\bK^H$-algebra $A$ split over $\bK$ with $\alpha_A = a$.
This class contains a unique (up to an isomorphism) division algebra $D$,
so $\langle a,H\rangle \in \bBM (G)$ can also
be
interpreted as 
an isomorphism class $[D]$ of division $\bK^H$-algebras, split over $\bK$ with $\alpha_D = a$ 

Now the extended Burnside ring will play the same role for the study 
of central simple algebras as the usual Burnside ring plays for the study
of fields: various number theoretic concepts become group homomorphisms 
from $\bBM (G)$ to abelian groups \cite{DoDo}. For instance, 
a zeta function $\zeta_D (z)$ of a division algebra $D$
extends to a group homomorphism to the meromorphic functions
$\zeta : \bBM(G) \rightarrow \bM (z)$: 
on basis elements $\zeta (\langle a,H\rangle ) = \zeta_D (z)$ where
$D$ is the division central $\bK^H$-algebra, split over $\bK$ with $\alpha_D = a$.

\section{Groupoids and $\mu_\bK$-decorated sets}

Over a field $\bK$, there is a bijection between
elements of $\mu_\bK (G)$ and isomorphism classes of central extensions
$$
1 \rightarrow
\bK^\times \rightarrow
\widetilde{G} \rightarrow
G \rightarrow 1 .
$$
The goal of this section is to observe
that
$\mu_\bK$-decorated sets admit a similar interpretation via groupoids.
Any $G$-set $X$ defines the action groupoid $\sG_X=G\times X$ 
over the base $X$.
The maps $\pi_1, \pi_2 : \sG_X \rightarrow X$ are
$\pi_1(g,x)=g\cdot x$ and $\pi_2(g,x)=x$.
The product $(g,x)(h,y)=(gh,y)$ is defined whenever
$\pi_2(g,x)=\pi_1(h,y)$.
A central extension of $\sG_X$ by  $\bK^\times$
is an exact sequence of groupoids
$$
1 \rightarrow
\bK^\times \times \Delta_X \rightarrow
\widetilde{\sG}_X \rightarrow
\sG_X \rightarrow 1 
$$
where $\bK^\times \times \Delta_X$
is 
a trivial groupoid on the diagonal $\Delta_X \subseteq X\times X$ \cite{MirkovicRumynin99}, i.e.,
$\pi_1, \pi_2 : \bK^\times \times \Delta_X \rightarrow \Delta_X$ are both
$\pi_1(g,x,x)=\pi_2(g,x,x)=(x,x)$ and
$(g,x,x)(h,x,x)=(gh,x,x)$.

\begin{lemma}
\label{groupoid}
There are natural bijections between the following sets:
\begin{mylist}
\item[(1)] isomorphism classes of finite $\mu_\bK$-decorated $G$-sets, 
\item[(2)] isomorphism classes of central extensions by $\bK^\times$ of $G$-action groupoids on finite sets.
\end{mylist}
\end{lemma}
\begin{proof}
Such central extensions are defined by central extensions of the diagonal groups $\sG_{x,x}=\pi_1^{-1}(x)\cap\pi_2^{-1}(x)$.
These diagonal groups are point stabilisers $G_x$ and their extensions are defined by $\pi_x\in \mu_\bK (G_x)$.

The equivariance assumption on frills is necessary for the existence of the central extension:
each $g\in G$ defines an automorphism of $\widetilde{\sG}$ by
$(h,x) \mapsto (ghg^{-1}, \,^gx)$. This automorphism gives an isomorphism between central extensions of $G_x$ and $G_{g\cdot x}$.
We leave it to the reader to check that the equivariance is 
sufficient for $\widetilde{\sG}$ to be well defined.

Thus, central extensions of action groupoids and $\mu_\bK$-decorated sets are defined by the same data, 
so there is an obvious natural isomorphism between the sets of isomorphism
classes of both.
\end{proof}

Furthermore, it is possible to write a presentation of 
$\bBM (G)$ in the language of central extension groupoids. We leave details to an interested reader.

\section{Module categories and $\mu_\bK$-decorated sets}
\label{sec_mod_cat}

To explain the final (in this paper) interpretation
of the generalised Burnside ring  $\bBM (G)$,
we need to contemplate the relation between a $G$-algebra $A$ and the skew group ring $A\ast G$.
We have already seen that $A-\Mod$ is a $G$-category.
What is about $A\ast G-\Mod$?
It is a (right) module category over $G-\Mod$.
This means there is an exact tensor product bifunctor
$$
\boxtimes: A\ast G-\Mod \  \times \ G-\Mod \rightarrow A\ast G-\Mod
$$
with associativity and unity natural transformations
$$
(M \boxtimes V) \boxtimes V^\prime \stackrel{\cong}{\longrightarrow} M \boxtimes (V \otimes V^\prime ), \ \ \
M \boxtimes \bK \stackrel{\cong}{\longrightarrow} M 
$$
where $\bK$ is the trivial $G$-module subject to the commutativity of the 
pentagon and triangle diagrams 
\cite{EGNO,Ostrik03}.
Both citations are comprehensive sources on module categories.
We will use their terminology and results freely in this section.

The tensor product $M\boxtimes V$ 
of an $A*G$-module $M$ and a $G$-module $V$ is just the usual tensor product $M\otimes V$ of $G$-modules with $A$ acting on the first component.
In fact, $A\ast G-\Mod$ is naturally equivalent (as a module category) to the module category $A-\Mod_G$ \cite{EGNO,Ostrik03}.
To construct the latter, $A$ is considered as an algebra in $G-\Mod$ and $A-\Mod_G$ is the category of $A$-modules in  $G-\Mod$.

Now we indulge in a philosophical digression:
the precise relation  
between $A-\Mod$ and $A\ast G-\Mod$ is of duality.
Lemma~\ref{G-eq} gives an equivalence between $A\ast G-\Mod$ and the category of equivariant objects in $A-\Mod$ with fixed equivariant structures.
The Cohen-Montgomery duality for actions tells us that $(A*G)\# (\bK G)^* \cong M_n(A)$
where $n$ is the order of $G$
\cite{CoMo}.
Thus, $A-\Mod$ is equivalent to $(A*G)\# (\bK G)^* -\Mod$ which is the category of $G$-graded $A*G$-modules.


\begin{lemma}
\label{AG_morita}
Let $A$ and $B$ be associative $G$-algebras.
The categories $A\ast G-\Mod$
and
$B\ast G-\Mod$
are equivalent as module categories over $G-\Mod$
if and only if there exists a nondegenerate $G$-equivariant 
Morita context $(A,B, \,_AM_B, \,_BN_A, \phi, \psi)$.
\end{lemma}
\begin{proof}
%
The category $A\ast G-\Mod$ is naturally equivalent to $A-\Mod_G$, the category of $A$-modules in  $G-\Mod$.
A nondegenerate $G$-equivariant 
Morita context is just a nondegenerate Morita context in $G-\Mod$.
Thus, the lemma is just a standard Morita theorem stated inside the category $G-\Mod$, for instance, our proof of Theorem~\ref{G_morita}
set in $G-\Mod$ instead of vector space but with the trivial group will do the job. 
\end{proof}

It is useful to introduce a more intuitive geometric language  \cite{Bezruk01, BezrukFO06}.
We can think of a $\mu_\bK$-decorated $G$-set $X$
as a $G$-Morita equivalence class $[A]$ of split semisimple $G$-algebras over $\bK$.
By Lemma~\ref{AG_morita}, the category $A\ast G-\Mod$ is canonically attached to $X$,
i.e. if $X=[A]$ and $X=[B]$ for different $G$-algebras gives equivalent categories.
We call it {\em the category of $G$-equivariant coherent sheaves on $X$} and denote
$\CG(X)$.
The rank of the Grothendieck group $K(\CG(X))$, equal to the number of irreducible objects in 
$A\ast G-\Mod$, is {\em an equivariant
Euler characteristic} $\mM(X)$ of the $\mu_\bK$-decorated $G$-set. 
This linearly extends to 
a function 
$\mM:\bBM(G)\rightarrow \mathbb{Z}$,
whose values are appended to tables 1 and 2. 

Some of the considerations can be repeated if $G$ is no longer finite but an algebraic group acting on a finite set $X$.
As the stabilisers of points are open, the finite component group $G/G_0$ acts on $X$.
We define a $\mu_\bK$-decorated $G$-set to be just  a $\mu_\bK$-decorated $G/G_0$-set. 
Now the category $A\ast G-\Mod$ consists only of those $A\ast G$-modules that are rational as $G$-modules.
Now Lemma~\ref{AG_morita} can be repeated in $G$-modules and  the category $\CG(X)$ is canonically attached to $X$.

A point $x = [N] \in X$ determines a minimal central idempotents $e_x\in A$ such that $e_x N = N$.
Using it, we define 
a stalk $M_x := e_x M$ and the support $\{x\in X \mid e_x M\neq 0\}$
of a sheaf $M$. This will be used in the next section.

Now we would like to discuss the relation of $\CG(X)$ to the module categories $H-\Mod_\eta$.
If $\eta\in \mu_\bK (H)$ and $H$ is a subgroup of a finite group $G$, the category $H-\Mod_\eta$ 
is the category of projective representations of $H$, affording the cocycle $\eta$
\cite{EGNO,Ostrik03}.

\begin{lemma}
\label{duality}
Let $X$ be a finite $\mu_\bK$-decorated $G$-set,
$G$ a finite group, 
$X_0\subseteq X$ a set of representatives of $G$-orbits.
Then the category $\CG(X)$ is equivalent to
$\oplus_{x\in X_0} G_x-\Mod_{\pi_x^{-1}}$ as a module category. 
\end{lemma}
\begin{proof}
The functor $\Phi : \oplus_{x\in X_0} G_x-\Mod_{\pi_x^{-1}} \rightarrow \CG(X)$
is constructed in two steps. First, we can associate a conjugate projective representation $V_x\in G_x-\Mod_{\pi_x^{-1}}$,  $x\in X$
to a formal sum $\oplus_{x\in X_0} V_{x}$.
It is done exactly as in the proof of Theorem~\ref{morita_decor}.
Now let $M_x$ be the simple $A$-module that corresponds to the point $x\in X$. We define
$$
\Psi (\oplus_{x\in X_0} V_{x}) = \oplus_{x\in X} M_x\otimes_\bK V_{x}
$$
with $A$ acting on the first components. $G_x$ acting on the tensor product $M_x\otimes_\bK V_{x}$ (N.B., the cocycles cancel, so $H_x$ acts linearly)
and elements of the transversal $T_x$ permuting the components in the orbit.

Its quasiinverse functor $\Psi : \CG(X) \rightarrow  \oplus_{x\in X_0} G_x-\Mod_{\pi_x^{-1}} $ is based on the canonical decomposition
$$
L = \oplus_{x\in X} M_x \otimes \Hom_A (M_x, L)
$$
of an $A\ast G$-module $L$ (N.B., $A$ is semisimple). Observe that $L$ is a linear representation of $G$, $M_x$ 
a projective representation of $G_x$ with the cocycle $\pi_x$, so 
$\Hom_A (M_x, L)$ is 
a projective representation of $G_x$ with the cocycle $\pi_x^{-1}$. Thus,
$$
\Psi (L) = \oplus_{x\in X_0} \Hom_A (M_x, L)
$$
is the quasiinverse functor. All the verifications are straightforward.
\end{proof}

It is interesting that Lemma~\ref{duality}
holds without any assumption on characteristic $p$ of the field $\bK$.
If $p$ does not divide $|G|$ then every indecomposable semisimple module category over $G-\Mod$
is equivalent to  $H-\Mod_{\eta}$ for some $H$, $\eta$
\cite[Th 3.2]{Ostrik03}.
Thus, $\CG(X)$ are all possible semisimple module categories.

Now if $p$ divides $|G|$ then $A\ast G$ can be semisimple or  not semisimple.
However, it is relatively semisimple over $G-\Mod$.
It would be interesting whether $\CG(X)$ constitute all possible relatively
semisimple module categories in this case. We avoid this difficulty by declaring a module category
{\em special} if it is equivalent to a direct sum of $H-\Mod_{\eta}$ as a module category.

\begin{theorem}
\label{main_bijection}
For a finite group $G$ there are natural bijections between the following sets:
\begin{mylist}
\item[(1)] isomorphism classes of finite $\mu_\bK$-decorated $G$-sets, 
\item[(2)] isomorphism classes of central extensions  by $\bK^\times$ of $G$-action groupoids of finite sets,
\item[(3)] $G$-Morita equivalence classes of semisimple split $G$-algebras,
\item[(4)] equivalence classes of special module categories over $G-\Mod$.
\end{mylist}
\end{theorem}
\begin{proof}
After Theorem~\ref{morita_decor}
and Lemmas~\ref{groupoid}, \ref{AG_morita} and \ref{duality},
the only thing left to prove is that
if 
$H-\Mod_{\eta}$ is equivalent to $H^\prime -\Mod_{\eta^\prime}$ as a module category
then $(H,\eta)$ is conjugate to $(H,\eta^\prime)$.
Let $X=G/H$, $X=G/H^\prime$ with frills $\pi_{gH}= g\eta^{-1} g^{-1}$,  $\pi_{gH^\prime}= g{\eta^\prime}^{-1} g^{-1}$.
Since $H-\Mod_{\eta}$ is equivalent to $\CG (X)$, $\CG(X)$ is equivalent to $\CG (X^\prime)$.
So $X$ must be isomorphic to $X^\prime$ as decorated sets.
If $\varphi : X^\prime \rightarrow X$ is an isomorphism 
and $\varphi (H^\prime) = g$ then
$g(H,\eta)q^{-1}=(H^\prime, \eta^\prime)$.
\end{proof}

Using Theorem~\ref{main_bijection}, one can
write a presentation of $\bBM(G)$ in the language of module categories. 
We leave it to an interested reader, making only one relevant observation.
Let $[\sM]\in \bBM(G)$
be the equivalence class of a special module category $\sM$.
Observe that if $\sM$ and $\sN$ are special module categories
as in Theorem~\ref{main_bijection} then the category of module functors
$\Fun(\sM,\sN)$ is a special module category and 
$$
[\Fun(\sM,\sN)]=[\sM]^\vee \cdot [\sN] .
$$
The remaining  sections of the paper are devoted to applications of Burnside rings.
An interesting group for the applications is the component group $A_\chi$ of a centraliser of a nilpotent element
(in a simple Lie algebra) \cite{Bezruk01,BezrukFO06}. 
The groups that occur as $A_\chi$ are symmetric groups $S_3$, $S_4$, $S_5$ and elementary abelian 2-groups
$C_2^n$. A feature of these groups is that the Schur multipliers $\mu (A)$ of their subgroups are elementary abelian 2-groups.
This implies that $[X]=[X]^\vee$, simplifying the calculations.
 
For instance, the number of simple objects in the module category 
$\Fun(\sM,\sN)$ over $A_\chi-\Mod$
is $\mM ([\sM][\sN])$. 
In the course of  a proof \cite[Th. 3]{BezrukFO06}, the authors 
show that for $[\sM],[\sN] \in \bBM(S_4)$ 
such that $\mM ([\sM][\sM])=\mM([\sN][\sN])=5$ and $\mM([\sM][\sN])=3$, either $[\sM][\sM]=\langle S_4\rangle$ or $[\sN][\sN] =\langle S_4\rangle$.  
This follows immediately from Table~\ref{tab:S4} since 
$\mM([\sM][\sM])=5$ implies  
$[\sM]\in \{ \langle S_3 \rangle, \langle S_4\rangle , \langle S_4^\prime\rangle \}$.

\section{Application: Kazhdan-Lusztig cells}
A Coxeter group $W$
admits three equivalence relations $\sim_L$, $\sim_R$ and $\sim_{LR}$. 
Equivalence classes of these relations are called left cells, right cells, and double cells correspondingly
\cite{Lus}.
The definition of $\sim_L$ involves chains of elements, whose lengths may grow. 
Although no explicit bound
on the lengths of elements is known, it is expected that $x\sim_Ly$ can be decided by an efficient algorithm
(cf., Casselman's Conjecture \cite{Gun}).

If $W$ is an affine Weyl group of a simple algebraic group $G^\vee$, cells admit a particularly revealing description.
To a double cell $C\subseteq W$
Lusztig' bijection associates a particular nilpotent coadjoint orbit $G\cdot \chi$ of the Langlands dual group $G$
(over $\bC$ or any algebraically closed field of good characteristic). 
Let $G_\chi$ be the reductive part of the stabiliser of $\chi$, 
$A_\chi = G_\chi / {G^0}_\chi$
its component group.
By Bezrukavnikov-Ostrik's theorem, the cell admits a base $\mu$-decorated $A_\chi$-set $Y_C$
\cite{Bezruk01}.

We refer an interested reader to the original Lusztig's paper \cite[Conj 10.5]{Lus} for a full definition of the base set
but one should be warned the sets there are not decorated and the term ``base set'' is not used. 
Here we list some of its properties, crucial for the further exposition here:
\begin{mylist}
\item[(1)] the permutation representation $\bC Y_C$ is isomorphic to the representation of $A_\chi$ on $H^* (\sB^\chi, \bC)$, the total cohomology
of the Springer fibre,
\item[(2)] there is a bijection between $C$ and the set of isomorphism classes of irreducible objects \newline in
$\CGc(Y_C \times Y_C)$.
\item[(3)] If $Y_C = \coprod_i Y_i$ where $Y_i$ are $A_\chi$-orbits then the left cells correspond to sheaves supported on various
$Y_C \times Y_i$ while the right cells correspond to sheaves on 
$Y_i \times Y_C$.
\end{mylist}

This information allows us to determine $Y_C$ uniquely if $A_\chi$ is cyclic. 
In particular, all Schur multipliers vanish in this
case and all the decorations on the set $Y_C$ must be trivial. 
If $A_\chi = S_3$ then it is not clear how to determine $Y_C$ 
explicitly but the decorations must be trivial as all Schur multipliers vanish. 
The remaining component possible component groups 
are $S_4$, $S_5$ and elementary abelian $2$-groups. 
The aim of this section is to compute $Y_C$ in the case of $A_\chi =S_4$.

This component group appears only in the type $F_4$ in the orbit $F_4 (a_3)$. 
The corresponding double cell is 
$$
C
= \{x \in W \; | \; x\sim_{LR} s_2s_3s_2s_3 \}
=\{x \in W \; | \; \mbox{{\bf a}}(x)=4\} 
$$
where $W$ is the affine Weyl group of the type $F_4$, {\bf a} is Lusztig's {\bf a}-function, $s_2$, $s_3$ are the two simple reflections connected
by the double arrow. 
The Green function
\cite{Sho} of  $F_4 (a_3)$ is
$$
(\chi_{12}q^4 +(\chi_{8,3}+\chi_{8,1})q^3+\chi_{9,1}q^2+\chi_{4,1}q+1)\Sigma_4 + 
(\chi_{9,3}q^4+\chi_{8,3}q^3+\chi_{2,3}q^2)\Sigma_{3,1}
+
(\chi_{6,2}q^4+\chi_{4,1}q^3)\Sigma_{2,2}
+
\chi_{1,3}q^4\Sigma_{2,1,1}
$$
where $\Sigma_\pi$ denotes the irreducible character of $S_4$ corresponding to a partition $\pi$, $\chi_{n,m}$ is an irreducible $n$-dimensional character of the finite Weyl
group $W_0$.of degree $m$, $q^k$ signifies that this component appears in degree $2k$ cohomology.
Essentially, the Green function records
$H^* (\sB^\chi, \bC)$ as a graded $A_\chi\times W_0$-module.

Let $\Omega : \bB (S_4) \rightarrow \Rep (S_4)$ be the natural homomorphism that assigns 
its permutation representation to an $S_4$-set. Let $\bB_+ (S_4)$ be the effective
part of the Burnside ring, i.e., the elements $[X]$ for actual $S_4$-sets.
The following lemma is checked by a straightforward calculation and left to the reader.

\begin{lemma}
\label{20_sols}
The equation
$$
\Omega ([X]) = 
42
\Sigma_4 + 
19 
\Sigma_{3,1}
+
10
\Sigma_{2,2}
+
\Sigma_{2,1,1}
$$
has 20 solutions in $\bB_+ (S_4)$:
$$
Y_\varepsilon = (15+\varepsilon)\langle S_4\rangle +(17-\varepsilon)\langle S_3\rangle +(9-\varepsilon)\langle D_8\rangle + \langle C_2 \rangle +\varepsilon \langle K_1\rangle ,
$$
$$ \
X_\varepsilon = (13+\varepsilon)\langle S_4\rangle +(19-\varepsilon)\langle S_3\rangle +(9-\varepsilon)\langle D_8\rangle + \langle C_4\rangle +\varepsilon \langle K_1\rangle 
$$
for various  $0 \leq \varepsilon \leq 9$.
\end{lemma}

These are 20 candidates for the base set $Y_C$.
Points in the orbits with stabilisers $S_4$, $D_8$ and $K_1$ 
may have non-trivial decorations, so the total number of candidate $\mu$-decorated sets 
is much bigger.
To advance further 
we need to know some explicit information about the cell itself.
More precisely, we need to know some elements in the 42 left cells contained in $C$.
At present, no publicly available software can compute cells.
However, 
we have managed to verify the 
following facts (stated as a proposition) on a computer.

\begin{prop} 
\label{prop_comp}
The following facts about the double cell
$C=\{x \in W (\widetilde{F_4})\; | \; \mbox{{\bf a}}(x)=4\}$ 
are true:
\begin{mylist}
\item[(1)] all left cells in $C$ contain at least 151 elements, 
\item[(2)] at least 30 cells in $C$ contain at least 175 elements, 
\item[(3)] the double cell $C$ contains at least 7400 elements.
\end{mylist}
\end{prop}

Proposition~\ref{prop_comp} can be verified on a computer by other research groups if they
wish. Hopefully, it could be done using some standard packages in future. It allows
us to pinpoint the base set of $C$ further.

\begin{theorem}
If Proposition~\ref{prop_comp} holds, 
then the base set $Y_C$ is one of the 8 sets listed in upper half of Table~\ref{tab:F4a3}.
\end{theorem}
\begin{proof}
Let $\overline{Y_C}$ be the underlying set of the decorated set $Y_C$.
It must be one of the twenty sets listed in Lemma~\ref{20_sols}.

Using (1) of Proposition~\ref{prop_comp}, 
we can rule out the case of  $[ \overline{Y_C}] = X_\varepsilon$ because 
one the left cells will contain
$\mM([Y_C] \cdot \langle C_4\rangle )= \mM(X_\varepsilon \cdot \langle C_4\rangle )
= \mM (24\langle C_4\rangle  + 9\langle H_2 \rangle  
+ 20\langle 1\rangle )= 24\times 4+9\times 2 +20 = 134 < 151$ elements.
Hence,  $[ \overline{Y_C}] = Y_\varepsilon$ with $0 \leq \varepsilon \leq 9$.

Notice that $\mM([ Y_C] \cdot \langle C_4\rangle )
=\mM(Y_\varepsilon \cdot \langle C_2\rangle )= \mM (60\langle H_2 \rangle  + 31\langle 1\rangle )= 60\times 2 +31 = 151$, so one of the left cells contains exactly 151 elements.
Moreover,
$(17-\varepsilon)$ further left cells contain exactly 
$\mM([Y_C] \cdot \langle S_3\rangle )
=\mM(Y_\varepsilon  \cdot \langle S_3\rangle )= 
\mM(32\langle S_3\rangle  + 28\langle C_2\rangle  + \langle 1\rangle ) = 32\times 3 + 28 \times 2 +1 =153$.
By (2) of Proposition~\ref{prop_comp}, at  most 12 left cells may have such a 
small number of elements. So,  $12 \geq 18-\varepsilon$ and $9 \geq \varepsilon \geq 6$.

To pinpoint extensions, we introduce 3 more variables to write
$$
Y_C = (15+\varepsilon-\alpha)\langle S_4\rangle +\alpha\langle S_4'\rangle +(17-\varepsilon)\langle S_3\rangle +(9-\varepsilon-\beta)\langle D_8\rangle +\beta\langle D_8'\rangle + \langle C_2 \rangle +(\varepsilon-\delta) \langle K_1\rangle +\delta\langle K_1'\rangle 
.
$$
Since $Y_C^\vee=Y_C$, the number of elements in $C$ is 
$$
\mM(Y_C \cdot Y_C)= 4\varepsilon^2 - 4\varepsilon\alpha - 12\varepsilon\gamma + 30\varepsilon 
+ 4\alpha^2 + 12\alpha\beta + 12\alpha\gamma - 114\alpha + 12\beta^2 + 12\beta\gamma - 198\beta + 12\gamma^2 - 144\gamma + 7084
.
$$
Using Matlab, we find 14 possible extended sets that could give at least 7400 elements in  the double cell. Results are summarised in table
\ref{tab:F4a3}. 
The 6 sets in the lower half of the table contain a cell with less than 151 elements, 
thus contradicting (1).
\end{proof}

Observe that the candidate sets come naturally in pairs, for instance, 
$[X] = 21\langle S_4\rangle +11\langle S_3\rangle +3\langle D_8\rangle + \langle C_2\rangle +6\langle K_1\rangle $ and 
$[Y] =21\langle S_4'\rangle +11\langle S_3\rangle +3\langle D_8'\rangle + \langle C_2\rangle +6\langle K_1'\rangle $.
In each pair $X\times X^\vee \cong Y \times Y^\vee$. 
Thus, if one set in a pair is a base set, so is the second set.
Since each pair contains a set with trivial decorations, we have established (subject to computer use in
Proposition~\ref{prop_comp}).

\begin{cor}
The cell $C$ admits an undecorated base set. 
\end{cor}

Our computer calculation establishes  
that certain elements are related by one of Kazhdan-Lusztig equivalences. 
At present, we do not know that 
the calculation exhausts all elements in the cell. 
However, 
the calculation indicates strongly 
that there are 11 cells of 153 elements. 
Thus, we can conclude (with a high degree
of confidence but not definite) 
that the base sets of the cell $C$ are
$$
\langle X\rangle  = 21\langle S_4\rangle +11\langle S_3\rangle +3\langle D_8\rangle + \langle C_2\rangle +6\langle K_1\rangle 
\mbox{ and }
\langle Y\rangle =21\langle S_4'\rangle +11\langle S_3\rangle +3\langle D_8'\rangle + \langle C_2\rangle +6\langle K_1'\rangle .
$$

\section{Application: reduced enveloping algebras}

Let $G$ be a simple simply-connected algebraic group over an algebraically closed field $\mathbb K$ of characteristic $p$ which is larger
than the Coxeter number of $G$. Let $\fg$ be its Lie algebra, $\chi\in \fg^\ast$ a nilpotent element, $U = U_\chi (\mg)$ the reduced
enveloping algebra. 
The finite dimensional algebra  $U$ splits into blocks  $U = \oplus_\lambda U^\lambda$ that are parametrised by the orbits
of the dual extended affine Weyl group $W^\prime = W_0\ltimes \Lambda$ on the weight lattice $\Lambda$ 
via $(w,\mu)\bullet\lambda = w(\lambda + \rho + p\mu)-\rho $ where $\rho$ is the half-sum of simple roots \cite{Jan}.
The reductive part of the stabiliser
$G_\chi$ acts on each $U^\lambda$
\cite{BMR}.
We are interested in determining the $\mu$-decorated $G_\chi$-set $Y^\lambda=\Irr (U^\lambda)$ for each $\lambda$.
As before, only the component group $A_\chi = G_\chi \big/ G_\chi^0$ acts on $Y^\lambda$, so it is a $\mu$-decorated $A_\chi$-set.

With our restriction on $p$, one can associate a parabolic subgroup $P=P(\lambda)$ (unique up to its type) 
to the weight $\lambda$ so that $\lambda$
is $P$-regular and $P$-unramified \cite{BMR}. Let $W(\lambda)$ be the corresponding parabolic subgroup in the finite Weyl group $W_0$.
Let $\Omega (Y^\lambda)$ be the permutation representation of $A_\chi$ over $\bC$. Then \cite{BMR, GoRo},
$$
\Omega (Y^\lambda) \cong H^* (G/P^\chi, \bC) \cong H^* (\sB^\chi, \bC)^{W(\lambda)}.
$$
In particular, $\Omega (Y^\lambda)$ depends only on the type of the parabolic. 
In fact, $Y^\lambda$ depends only on the type of the parabolic because the translation functor
within the same wall is a $G_\chi$-equivalence \cite{BMR,Jan}.
\begin{hyp}
If $P(\nu)\subseteq P(\lambda)$ then there exists an $A_\chi$-subset $Y_0^\lambda \subseteq Y^\lambda$
and a surjective morphism $Y_0^\lambda \rightarrow Y^\nu$ of $A_\chi$-sets.
\end{hyp}
This morphism should be performed by the translation to the wall.
We are happy to leave it as a conjecture at this point.
It will be explained elsewhere.

Now we specialise the set-up to $\fg$ of the type $F_4$ and $\chi$ of the type $F_4 (a_3)$, i.e.,
$\chi$ belongs to the only orbit with the component group $S_4$. It corresponds to the cell $C$ of the previous section under Lusztig's bijection.
The underlying undecorated $S_4$-sets of the sets $Y^\lambda$ are listed in Table~\ref{tab:F4}.
The left column contains the list of the types of parabolic subalgebras.
The middle column describes the representation $\Omega (Y^\lambda )$ of
$S_4$ by listing the multiplicities of irreducible constituents.

Now the right column describes the sets. The first five most degenerate parabolic types can be computed uniquely without the use
of the hypothesis. Indeed,
$$
\Omega \langle S_3\rangle  = \Sigma_4 + \Sigma_{3,1}
\mbox{ and } 
\Omega \langle S_4\rangle  = \Sigma_4
$$ 
are the only permutation characters of $S_4$ that have only $\Sigma_4$ and $\Sigma_{3,1}$ as constituents.

The second two types can be computed using the hypothesis.
Besides $\langle S_3\rangle$ and $\langle S_4\rangle$
there are four
$S_4$-sets without $\Sigma_{1,1,1,1}$ in 
the permutation representation:
$$
\Omega \langle C_2 \rangle  = \Sigma_4 + 2 \Sigma_{3,1}+ \Sigma_{2,2}+ \Sigma_{2,1,1}, \ 
\Omega \langle C_4 \rangle  = \Sigma_4 + \Sigma_{2,2} + \Sigma_{2,1,1}, \ 
\Omega \langle K_1\rangle  = \Sigma_4 + \Sigma_{3,1} + \Sigma_{2,2}, \ 
\Omega \langle D_8\rangle  = \Sigma_4 + \Sigma_{2,2}
$$
The $S_4$-set for $W(1,2)$ can be degenerated to the sets for $W(1,2,4)$, hence it is at least $3\langle S_4\rangle +4\langle S_3\rangle $.
The rest of the set has the permutation character
$4\Sigma_4 + 5 \Sigma_{3,1}+ \Sigma_{2,2}+ \Sigma_{2,1,1}$
leaving the only possibility of $\langle C_2 \rangle +3\langle S_3\rangle $. Similarly, the set for $W(3,4)$ degenerates to the set for $W(1,3,4)$, so it is at least
$6\langle S_4\rangle +\langle S_3\rangle $, leaving the only possibility of $9\langle S_4\rangle +\langle S_3\rangle +\langle D_8\rangle $.

The remaining five sets cannot be uniquely determined by this method. 
One needs to know how many times $\langle K_1\rangle $ appears in the set. 
We make this multiplicity into a parameter and list the remaining sets. We expect all the frills on all $Y^\lambda$ to be trivial
and $\varepsilon =6$ in the light of the following Lusztig's conjecture \cite{Lus2}:
\begin{con}
For each $G$ and $\chi$
\begin{mylist}
\item[(1)] the frills of $Y^\lambda$ are trivial, 
\item[(2)] $Y^0$ is a base set of the double cell in the dual affine Weyl group of $G$ that corresponds to the orbit of $\chi$ under Lusztig's bijection.
\end{mylist}
\end{con}

\section{Appendix: Tables}

\begin{table}[ht]
\caption{The extended table of marks of $S_4$.} 
\label{tab:S4}
{\tiny
\begin{tabular}{@{\extracolsep{-4pt}}c|ccccccccccc|ccccc|c}
                                           &&&&&&&&&&&&&&&&&${\mathfrak M}$\\\hline
                                      $1$&24&&&&&&&&&&&&&&&&1\\
                                    $H_2$&12&4&&&&&&&&&&&&&&&2\\
                                    $C_2$&12&0&2&&&&&&&&&&&&&&2\\
                                    $C_3$&8&0&0&2&&&&&&&&&&&&&3\\
                                    $C_4$&6&2&0&0&2&&&&&&&&&&&&4\\
                                    $S_3$&4&0&2&1&0&1&&&&&&&&&&&3\\
                                    $K_1$&6&2&2&0&0&0&2&&&&&2&&&&&4\\
                                    $K_2$&6&6&0&0&0&0&0&6&&&&0&6&&&&4\\
                                    $D_8$&3&3&1&0&1&0&1&3&1&&&1&3&1&&&5\\
                                    $A_4$&2&2&0&2&0&0&0&2&0&2&&0&2&0&2&&4\\
                                    $S_4$&1&1&1&1&1&1&1&1&1&1&1&1&1&1&1&1&5\\\hline
                             $K_1^\prime$&6&2&2&0&0&0&2&&&&&$-2$&&&&&1\\
                             $K_2^\prime$&6&6&0&0&0&0&0&6&&&&0&$-6$&&&&1\\
                             $D_8^\prime$&3&3&1&0&1&0&1&3&1&&&1&3&$-1$&&&2\\
                             $A_4^\prime$&2&2&0&2&0&0&0&2&0&2&&0&$-2$&0&$-2$&&3\\
                             $S_4^\prime$&1&1&1&1&1&1&1&1&1&1&1&1&1&$-1$&1&$-1$&3
\end{tabular}
}
\end{table}

\begin{table}[ht]
\caption{The extended table of marks of $S_5$. } 
\label{tab:S5}
{\tiny
\begin{tabular}{@{\extracolsep{-5.2pt}}c|lcccccccccccccccccc|cccccccc|c}
                                    &&&&&&&&&&&&&&&&&&&&&&&&&&&&${\mathfrak M}$\\\hline
                               $1$&120&&&&&&&&&&&&&&&&&&&&&&&&&&&1\\
                             $H_2$&60&4&&&&&&&&&&&&&&&&&&&&&&&&&&2\\
                             $C_2$&60&0&6&&&&&&&&&&&&&&&&&&&&&&&&&2\\
                             $C_3$&40&0&0&4&&&&&&&&&&&&&&&&&&&&&&&&3\\
                             $C_4$&30&2&0&0&2&&&&&&&&&&&&&&&&&&&&&&&4\\
                             $C_5$&24&0&0&0&0&4&&&&&&&&&&&&&&&&&&&&&&5\\
                             $S_3$&20&0&6&2&0&0&2&&&&&&&&&&&&&&&&&&&&&3\\
                             $H_6$&20&4&0&2&0&0&0&2&&&&&&&&&&&&&&&&&&&&3\\
                  $C_3 \times C_2$&20&0&2&2&0&0&0&0&2&&&&&&&&&&&&&&&&&&&6\\
                          $D_{10}$&12&4&0&0&0&2&0&0&0&2&&&&&&&&&&&&&&&&&&4\\
                             $K_1$&30&2&6&0&0&0&0&0&0&0&2&&&&&&&&&2&&&&&&&&4\\
                             $K_2$&30&6&0&0&0&0&0&0&0&0&0&6&&&&&&&&0&6&&&&&&&4\\
                          $H_{20}$&6&2&0&0&2&1&0&0&0&1&0&0&1&&&&&&&0&0&&&&&&&5\\
                             $D_8$&15&3&3&0&1&0&0&0&0&0&1&3&0&1&&&&&&1&3&1&&&&&&5\\
                             $A_4$&10&2&0&4&0&0&0&0&0&0&0&2&0&0&2&&&&&0&2&0&2&&&&&3\\
                   $S_3\times C_2$&10&2&4&1&0&0&1&1&1&0&2&0&0&0&0&1&&&&2&0&0&0&1&&&&6\\
                             $S_4$&5&1&3&2&1&0&2&0&0&0&1&1&0&1&1&0&1&&&1&1&1&1&0&1&&&5\\
                             $A_5$&2&2&0&2&0&2&0&2&0&2&0&2&0&0&2&0&0&2&&0&2&0&2&0&0&2&&5\\
                      $S_5$&1&1&1&1&1&1&1&1&1&1&1&1&1&1&1&1&1&1&1&1&1&1&1&1&1&1&1&7\\\hline
                      $K_1^\prime$&30&2&6&0&0&0&0&0&0&0&2&&&&&&&&&$-2$&&&&&&&&1\\
                      $K_2^\prime$&30&6&0&0&0&0&0&0&0&0&0&6&&&&&&&&0&$-6$&&&&&&&1\\
                      $D_8^\prime$&15&3&3&0&1&0&0&0&0&0&1&3&0&1&&&&&&1&3&$-1$&&&&&&2\\
                      $A_4^\prime$&10&2&0&4&0&0&0&0&0&0&0&2&0&0&2&&&&&0&$-2$&0&$-2$&&&&&3\\
        ${S_3\times C_2}^\prime$&10&2&4&1&0&0&1&1&1&0&2&0&0&0&0&1&&&&$-2$&0&0&0&$-1$&&&&3\\
                   $S_4^\prime$&5&1&3&2&1&0&2&0&0&0&1&1&0&1&1&0&1&&&1&1&$-1$&1&0&$-1$&&&3\\
              $A_5^\prime$&2&2&0&2&0&2&0&2&0&2&0&2&0&0&2&0&0&2&&0&$-2$&0&$-2$&0&0&$-2$&&4\\
                      $S_5^\prime$&1&1&1&1&1&1&1&1&1&1&1&1&1&1&1&1&1&1&1&1&1&$-1$&1&$-1$&1&1&$-1$&5
\end{tabular}
}
\end{table}

\bigskip

\begin{table}[ht]
\caption{Candidate base $S_4$-sets for cell $F_4 (a_3)$.} 
\label{tab:F4a3}
{\tiny
\begin{tabular}{@{\extracolsep{-4pt}}c|l|c|r}
$\varepsilon$ & \mbox{ Set } & \mbox{double cell size} & \mbox{partition into left cell} 
\\\hline
&&& \\
6 & $21\langle S_4\rangle +11\langle S_3\rangle +3\langle D_8\rangle + \langle C_2\rangle +6\langle K_1\rangle $ & 7408 & $(151,153^{11},179^{21},193^3,206^6)$ \\
6 & $21\langle S_4'\rangle +11\langle S_3\rangle +3\langle D_8'\rangle + \langle C_2\rangle +6\langle K_1'\rangle $ & 7408 & $(151,153^{11},179^{21},193^3,206^6)$ \\
7 & $22\langle S_4\rangle +10\langle S_3\rangle +2\langle D_8\rangle + \langle C_2\rangle +7\langle K_1\rangle $ & 7490 & $(151,153^{10},180^{22},193^2,209^7)$ \\
7 & $22\langle S_4'\rangle +10\langle S_3\rangle +2\langle D_8'\rangle + \langle C_2\rangle +7\langle K_1'\rangle $ & 7490 & $(151,153^{10},180^{22},193^2,209^7)$ \\
8 & $23\langle S_4\rangle +9\langle S_3\rangle +\langle D_8\rangle + \langle C_2\rangle +8\langle K_1\rangle $ & 7580 & $(151,153^{9},181^{23},193,212^8)$ \\
8 & $23\langle S_4'\rangle +9\langle S_3\rangle +\langle D_8'\rangle + \langle C_2\rangle +8\langle K_1'\rangle $ & 7580 & $(151,153^{9},181^{23},193,212^8)$ \\
9 & $24\langle S_4\rangle +8\langle S_3\rangle +\langle C_2\rangle +9\langle K_1\rangle $ & 7678 & $(151,153^{8},182^{24},215^9)$ \\
9 & $24\langle S_4'\rangle +8\langle S_3\rangle +\langle C_2\rangle +9\langle K_1'\rangle $ & 7678 & $(151,153^{8},182^{24},215^9)$ \\
&&& \\
\\\hline
&&& \\
8 & $22\langle S_4\rangle +\langle S_4'\rangle +9\langle S_3\rangle +\langle D_8\rangle + \langle C_2\rangle +8\langle K_1\rangle $ & 7438 & $(110,151,153^{9},179^{22},190,209^8)$ \\
8 & $22\langle S_4'\rangle +\langle S_4\rangle +9\langle S_3\rangle +\langle D_8'\rangle + \langle C_2\rangle +8\langle K_1'\rangle $ & 7438 & $(110,151,153^{9},179^{22},190,209^8)$ \\
9 & $24\langle S_4\rangle +8\langle S_3\rangle +\langle C_2\rangle +8\langle K_1\rangle +\langle K_1'\rangle $ & 7438 & $(95,151,153^{8},179^{24},209^8)$ \\
9 & $24\langle S_4'\rangle +8\langle S_3\rangle +\langle C_2\rangle +8\langle K_1'\rangle +\langle K_1\rangle $ & 7438 & $(95,151,153^{8},179^{24},209^8)$ \\
9 & $23\langle S_4\rangle +\langle S_4'\rangle +8\langle S_3\rangle +\langle C_2\rangle +9\langle K_1\rangle $ & 7532 & $(109,151,153^{8},180^{23},212^9)$ \\
9 & $23\langle S_4'\rangle +\langle S_4\rangle +8\langle S_3\rangle +\langle C_2\rangle +9\langle K_1'\rangle $ & 7532 & $(109,151,153^{8},180^{23},212^9)$ \\
\end{tabular}
}
\end{table}

\bigskip

\begin{table}[ht]
\caption{$S_4$-sets from parabolic blocks of $U_\chi$ with $\chi$ of type $F_4 (a_3)$.} 
\label{tab:F4}
{\tiny
\begin{tabular}{@{\extracolsep{-4pt}}c|ccccc|l}
& $\Sigma_4$ & $\Sigma_{3,1}$ & $\Sigma_{2,2}$ & $\Sigma_{2,1,1}$ & $\Sigma_{1,1,1,1}$ & \\\hline
                                      
$W(1,2,3,4)$ & 1 & 0 & 0 & 0 & 0 & $\langle S_4\rangle $\\
$W(1,2,3)$ & 3 & 2 & 0 & 0 & 0 & $\langle S_4\rangle +2\langle S_3\rangle $\\
$W(1,2,4)$ & 7 & 4 & 0 & 0 & 0 & $3\langle S_4\rangle +4\langle S_3\rangle $\\
$W(2,3,4)$ & 3 & 0 & 0 & 0 & 0 & $3\langle S_4\rangle $\\ 
$W(1,3,4)$ & 7 & 1 & 0 & 0 & 0 &  $6\langle S_4\rangle +\langle S_3\rangle $ \\ \hline

$W(1,2)$ & 11 & 9 & 1 & 1 & 0 & $3\langle S_4\rangle +7\langle S_3\rangle + \langle C_2\rangle  $ \\
$W(3,4)$ & 11 & 1 & 1 & 0 & 0 & $9\langle S_4\rangle +\langle S_3\rangle +\langle D_8\rangle $ \\ \hline

$W(1,3)$ & 15 & 6 & 2 & 0 & 0 &  $(7+\alpha)\langle S_4\rangle +(6-\alpha)\langle S_3\rangle +(2-\alpha)\langle D_8\rangle +\alpha \langle K_1\rangle , \ \alpha\leq 2$ \\
$W(2,3)$ & 10 & 4 & 2 & 0 & 0 & $(4+\beta)\langle S_4\rangle +(4-\beta)\langle S_3\rangle +(2-\beta)\langle D_8\rangle +\beta \langle K_1\rangle , \ \beta\leq 2 $ \\
$W(1)$ & 25 & 14 & 5 & 1 & 0 & $(8+\gamma)\langle S_4\rangle +(12-\gamma)\langle S_3\rangle +(4-\gamma)\langle D_8\rangle + \langle C_2\rangle +\gamma \langle K_1\rangle , \ \max(\alpha,\beta)\leq\gamma\leq 2$ \\
$W(3)$ & 25 & 8 & 5 & 0 & 0 & $(12+\delta)\langle S_4\rangle +(8-\delta)\langle S_3\rangle +(5-\delta)\langle D_8\rangle +\delta \langle K_1\rangle , \ \max(\alpha,\beta)\leq\delta\leq 4$ \\
$W(\emptyset)$ & 42 & 19 & 10 & 1 & 0 &  $(15+\varepsilon)\langle S_4\rangle +(17-\varepsilon)\langle S_3\rangle +(9-\varepsilon)\langle D_8\rangle + \langle C_2\rangle +\varepsilon \langle K_1\rangle , \ \max(\gamma,\delta)\leq \varepsilon \leq 8$ \\
\end{tabular}
}
\end{table}

\end{document}